\documentclass[11pt]{article}
\usepackage{amsmath}

\usepackage{amsfonts}
\usepackage{amssymb}
\usepackage{euscript}
\newenvironment{proof}{\noindent {\it Proof.~~}\ }{\  \rule{1mm}{2mm}\medskip}

\newenvironment{proof*}{\noindent {\it Proof.~~}\ }{}

\newtheorem{theorem}{Theorem}
\newtheorem{lemma}[theorem]{Lemma}
\newtheorem{corollary}[theorem]{Corollary}
\newtheorem{proposition}[theorem]{Proposition}

\newtheorem{theirtheorem}{Theorem}

\def\G{\partial}

\newcommand{\oa}{\overline{A}}

\begin{document}
\title{Extensions of the  Moser-Scherck-Kemperman-Wehn Theorem }
\author{ {Y. O. Hamidoune}\thanks{
UPMC Univ Paris 06,
 E. Combinatoire, Case 189, 4 Place Jussieu,
75005 Paris, France.}
}

\maketitle

%11B60, 11B34, 20D60.

\begin{abstract} Let $\Gamma =(V,E)$  be a  reflexive
relation having a transitive group of automorphisms and  let $v\in V.$
Let $F$  be a  subset of $V$ with $F\cap \Gamma ^-(v)=\{v\}$.
\begin{itemize}
  \item[(i)] If $F$  is finite, then
$| \Gamma (F)\setminus F|\ge
|\Gamma (v)|-1.$
\item [(ii)]If $F$  is cofinite, then
$| \Gamma (F)\setminus F|\ge
|\Gamma ^- (v)|-1.$
\end{itemize}

In particular,
let $G$  be group, $B$ be a finite subset of $G$
and let $F$ be a finite or a cofinite subset of $G$ such that  $F\cap  B^{-1}=\{1\}$.
Then
$| (FB)\setminus F|\ge
|B|-1.$
The last result (for $F$ finite), is  famous  Moser-Scherck-Kemperman-Wehn Theorem. Its   extension
to cofinite subsets seems new. We give also few applications.

 \end{abstract}

\section{Introduction}

A problem of Moser solved also by Scherck in \cite{sch} states that in ablian group $G$ two finite subsets $A,B$ with $|A\cap B^{-1}|=\{1\}$ verifies the Cauchy-Davenport inequality: $|AB|\ge |A|+|B|-1.$ The validity of this result in the non-abelian case was proved by Kemperman \cite{kempcompl}, mentioning that the result  was independently proved
by When \cite{kempcompl}. In our work,   $A$ could be  infinite. So we shall use the formulation $|AB\setminus A|\ge |B|-1.$

All the known proofs of the Moser-Scherck-Kemperman-Wehn Theorem, used additive transforms. In this work, we obtain a completely different proof.
The Moser-Scherck-Kemperman-Wehn Theorem has important applications
in Number Theory and the reader may find several applications of this beautiful result to the Theory of Non-unique factorization in the text book of Geroldinger-Halter-Koch \cite{gerlodinger}. Recall that this result  is used among other tools by Olson
\cite{olsonaa},  to prove that a  subset $S$ of a finite group $G$ with $|S|\ge 3\sqrt{|G|}$ contains a non-empty subset with a product (under some ordering) $=1$.
The Moser-Scherck-Kemperman-Wehn Theorem is a basic tool in the proof by Gao
that a sequence of elements of an abelian  group $G$ with length
$|G|+d(G)$ contains a $|G|$--sub-sequence summing to $0,$ where $d(G)$ is the maximal size of a sequence of elements of $G$ having no non-empty zero-sum subsequence \cite{gaotnd}. A recent generalization of Gao Theorem is contained in \cite{HweightD}. Notice that the Moser-Scherck-Kemperman-Wehn Theorem
implies easily the following result of Shepherdson \cite{sheph}:

 For every  nonempty subset $S$ of a finite abelian group $G,$
there are $s_1, \cdots , s_k\in S$ such that $k\leq \lceil \frac{|G|}{|S|}\rceil $ and  $\sum _{1\leq i \leq k} s_i=0.$

Some of the above results are related to questions investigated in Combinatorics.
Let $\Gamma=(V,E)$  be a loopless finite (directed) graph  with  $\min \{|\Gamma (x)| : x\in V\}\ge 1$. It is well known that $D$ contains a directed cycle. The smallest cardinality  of such a cycle is called the girth of $D$ and will be denoted by $g(D)$.
In 1978,  Caccetta and H\"aggkvist \cite{CH}  conjectured that
 $$|V|\ge \min (d^+x: x\in V)(g(D)-1)+1.$$

 This conjecture is still largely open. The reader may find references and  results about this question in \cite{Bondy}.

 This conjecture were proved by the author for graphs with a transitive group of automorphisms in  \cite{HEJC}. This result applied to Cayley graphs shows the validity of Shepherdson's  result for all finite groups. Unfortunately we were not aware at that moment of Shepherdson's result. Our proof  \cite{HEJC} is based on the properties of atoms of a finite graph \cite{HATOM} and the Dirac-Menger's Theorem \cite{Bondy2,Nat,tv}.
 Another proof of the
 Caccetta and H\"aggkvist Conjecture for the last class of graphs,  based
 on  the Moser-Scherck-Kemperman-Wehn Theorem  and the representation of point-transitive graphs as coset graphs, was given by Nathanson  in \cite{Nat0}.
 A third proof, using atoms but avoiding the Dirac-Menger's  was given later by the author in \cite{Hspheres}.

In an e-mail cited by Mader \cite{mader},   Seymour suggested  the possibility of the existence of  $\min \{|\Gamma (x)| : x\in V\}\ge 1$ directed cycles  having pairwise a fixed vertex in the intersection.
 Such a property implies clearly the Caccetta and H\"aggkvist \cite{CH}  Conjecture.
 Mader obtained a positive answer
if $\Gamma $ has a transitive group of automorphisms\cite{mader}.

Generalizing all the above results, we prove the following:

Let $\Gamma =(V,E)$  be a  reflexive
relation having a transitive group of automorphisms. Let $v\in V$
and let $F$  be a  subset of $V$ with $F\cap \Gamma ^-(v)=\{v\}$. Then
\begin{itemize}
  \item[(i)] If $F$  is finite, then
$| \Gamma (F)\setminus F|\ge
|\Gamma (v)|-1.$
\item [(ii)]If $F$  is cofinite, then
$| \Gamma (F)\setminus F|\ge
|\Gamma ^- (v)|-1.$
\end{itemize}

In particular, for  i
  an integer $j\ge 1$ with  $\Gamma ^{j-1}(v)\cap \Gamma ^{-}(v)=\{v\}$, we have
 $
|\Gamma ^{j}  (v)|\ge | \Gamma ^{j-1}  (v)| +
|\Gamma (v)|-1.$

The finite case of the  last result is proved by Mader in \cite{mader}.
The same conclusion was obtain under the stronger hypothesis $\Gamma ^j(v)\cap \Gamma ^{-}(v)=\{v\}$ in \cite{Hspheres}. This weak form is enough to show Caccetta and H\"aggkvist Conjecture.

Let $G$  be group and let $B$ be a finite subset of $G.$  Applying this result to the Cayley graph  defined on $G$ by $B,$
we obtain Moser-Scherck-Kemperman-Wehn Theorem and its extension
to cofinite subsets:

Let $F$ be a finite or a cofinite subset of $G$ such that  $F\cap  B^{-1}=\{1\}$.
Then
$| (FB)\setminus F|\ge
|B|-1.$

Our extension to the cofinite case allows a very short proof for the
following result, proved by Losonczy \cite{los} in the abelian case and by    Eliahou-Lecouvey \cite{el} in the general case:

If $1\notin B,$ then here is a permutation $\sigma$ of $B$ such that $x\sigma (x)\notin B$ for every
$x\in B.$

We conclude the paper by a simple proof of Mader's result. According to our approach, Seymour's question may be formulated as follows:

Under which condition, a graph
satisfies the Moser-Scherck-Kemperman-Wehn property?

Our proofs  require properties, developed in Sections 3 and 4, of  weak atoms
and Moser sets. The proofs we obtain for these properties  are easier  than the proofs of the corresponding properties of the  atoms \cite{HATOM}.

The present work is essentially self-contained. Our proof of Mader's Theorem
requires the Dirac-Menger's Theorem. Also our proof of a result due  Losonczy and    Eliahou-Lecouvey requires K\"onig-Hall's Theorem.  Let us state below these two results:

\begin{theirtheorem} ( The Dirac-Menger's Theorem  \cite{Bondy2}) \label{menger}

Let $\Gamma=(V,E)$ be a finite reflexive graph Let $k$ be a
nonnegative integer. Let $x,y\in V$ such that $y\notin \Gamma (x).$ If $|\G (X)|\ge k$ for every subset $X$ with $x\in X$ and   $y\notin \Gamma (X)$,
then there are $k$  disjoint directed paths from $\Gamma (x)$ to $\Gamma ^-(y)$.
\end{theirtheorem}

\begin{theirtheorem} \label{kh} (K\"onig-Hall's Theorem \cite{Bondy2})
Let ${\Phi}\subset V\times W$ be a relation
with $|V|=|W|<\infty$. Then  the following conditions  are equivalent:
 \begin{itemize}
   \item There is a bijection $\phi :V\mapsto W$ with $\phi (x)\in {\Phi} (x),$ for every $x\in V.$
   \item  $|{\Phi}(Y)|\ge |Y|,$ for every  subset $Y$ of $V.$
 \end{itemize}

\end{theirtheorem}

The reader may find proofs of the Dirac-Menger's Theorem in the text books \cite{Nat,tv} and some applications of this result to Additive Number Theory. Notice that the K\"onig-Hall's Theorem follows by applying
the Dirac-Menger's Theorem to the graph obtained from $\Phi$ by adding two distinct vertices one dominating $V$ and the other dominated par $W.$
\section{Some Terminology}

Let $\Gamma =(V,E)$ be a  relation and let $X\subset V.$
The subrelation $\Gamma [X]$ induced on $X$ is by definition
$(X,(X\times X)\cap E).$
 A function $f : V  \longrightarrow V$ will be called a {\em homomorphism }if for all $x\in V$, we have
$\Gamma (f(x))=f(\Gamma (x))$. A bijective homomorphism is called an {\em automorphism}. The relation $\Gamma$ will be called {\em locally-finite}  if for all $x\in V,$ $|\Gamma (x)|$ and $|\Gamma ^-(x)|$ are finite.

The relation $\Gamma$ will be called {\em point-transitive}  if for all $x,y\in V,$ there is an automorphism $f$
such that $y=f(x)$.
Clearly a point-transitive relation is regular. Let  $S$  be a
subset of   $G$. The relation $(G,E),$ where  $ E=\{ (x,y) : x^{-1}y \
\in S \}$ is called a {\it Cayley relation}.  It will  be denoted by
$\mbox{Cay} (G,S)$.

Let $\Gamma =\mbox{Cay} (G,S)$   and  let   $F \subset G $.
Clearly
 $\Gamma (F)=FS ,$ where $FS=\{xy: x\in F \ \mbox{and}\ y\in S\}$.
 Cayley  graphs are clearly point-transitive.

 We identify graphs (directed graphs) and their relations.
The reader may replace everywhere  the term "relation" by "graph".

 We shall write
 $$\partial _{\Gamma}(X)= \Gamma (X)\setminus X .$$ We also write
 $$\nabla _{\Gamma}(X)= V\setminus \Gamma (X) .$$

When the context is clear the reference to $\Gamma$ will be omitted.

Let $F$  be a subset of  $V.$
Clearly $V=F\cup \nabla(F)\cup \partial (F)$ is a partition.
Notice that  $\partial^{-}(\nabla(F))\cap F=\emptyset,$ otherwise there exist $z\in F$ such that  $\Gamma (z)\cap (\nabla(F))\neq \emptyset$, a contradiction. Hence
\begin{equation}\label{id0}
\partial^{-}(\nabla(F))\subset \partial (F).\end{equation}

The last observation, used extensively in the isoperimetric method \cite{Hast}, contains a useful duality. The next lemma contains a useful sub-modular relation:

\begin{lemma} \cite{Hast}\label{partialsub}{Let $\Gamma =(V,E)$ be a  locally finite   reflexive
graph. Let   $X,Y$     be finite  subsets. Then

 \begin{equation}\label{submodularity}
|\partial  (X \cup Y)|+|\partial  (X \cap Y)|\le |\partial  (X )|+|\partial  ( Y)|
\end{equation}}
\end{lemma}

\begin{proof}

 Observe that
\begin{eqnarray*}
|\Gamma (X\cup Y)|&=&|\Gamma (X)\cup \Gamma (Y)|\\
&=&|\Gamma (X)|+|\Gamma (Y)|-|\Gamma(X)\cap \Gamma (Y)|\\
&\le& |\Gamma (X)|+|\Gamma (Y)|-|\Gamma(X\cap Y)|
\end{eqnarray*}
The result follows now by subtracting the equation $|X\cup Y|=|X|+|Y|-|X\cap Y|$.
\end{proof}

\section{Weak atoms}
 The reader  interested only in the finite case may ignore this section.
 Let $\Gamma =(V,E)$ be a locally finite  reflexive relation.
 We define the {\em weak connectivity} of $\Gamma$ as
\begin{equation}  \label{eq:kappa}
\kappa  (\Gamma)=\min  \{|\partial (X)|\  : \ \ 1\le |X|<\infty \}.
\end{equation}

 A subset $X$ achieving the  minimum in  (\ref{eq:kappa}) is called a
{\em weak fragment} of $\Gamma$. A weak fragment with minimum cardinality
 will be called a {\em weak atom}.

\begin{proposition}  {
Let $\Gamma =(V,E)$ be a
 locally finite   relation. If $F_1$ and $F_2$ are weak
 fragments with a nonempty intersection, then $F_1\cup F_2$ and $F_1\cap F_2$ are weak fragments. In particular, two distinct weak atoms are disjoint. Assume furthermore that $\Gamma$ is point--transitive and
let $A$ be an atom of $\Gamma$. Then the weak atoms induce a partition of $V$. Moreover the subrelation $\Gamma [A]$ induced on $A$ is a point-transitive relation, for any weak atom $A$.
\label{atom}} \end{proposition}

\begin{proof}
By (\ref{submodularity}),
$2\kappa (\Gamma)\le |\partial  (F_1 \cap F_2)|+|\partial  (F_1 \cup F_2)|\le |\partial  (F_1 )|+|\partial ( F_2)|\le 2\kappa (\Gamma).$
Let $A$ be a weak atom and take $v\in A$. For every $y\in V,$ there is an automorphism $\phi$ such that $\phi (v)=y.$ Hence $y$ belongs to the atom $\phi (A).$ In particular, the set of atoms is a partition of $V.$ If $y\in A,$ then $\phi(A)=A,$ and therefore $\phi/A$ is an automorphism of $\Gamma[A]$ with $\phi/A(v)=y$. Thus $\Gamma [A]$  is a point-transitive relation.
\end{proof}

\section{Moser sets}
The formalism we develop here is motivated from one side by the Moser-Scherck-Kempermann-Wehn Theorem and by the notion of vertex-fragment
(a possibly empty set) in finite graphs considered by Mader in \cite{mader}  from the other side. Our Moser sets are related to Mader's
vertex-fragments. But our concept is never empty, works in the infinite case too and leads to easier proofs. In particular, we do not need a duality between positive and negative vertex-fragments, in the spirit of the one introduced in \cite{HATOM}, used extensively in Mader arguments. Also the Dirac-Menger's Theorem,
present in beginning of Mader's formalism, is not needed to prove our main result, generalizing of the Moser-Scherck-Kempermann-Wehn Theorem. We use it only to give a simple proof of Mader's result.

 Let $\Gamma =(V,E)$ be a locally finite  reflexive relation and let $v\in V$.
A set $F$ is said to be a $v$--Moser set if $\Gamma ^-(v)\cap F=\{v\}.$ Put
$$\kappa _{\Gamma}(v)=\min \{|\partial (X)| : X \ \mbox{is a} \ v \mbox{--Moser set}\}.$$
A Moser set $X$ with $|\partial (X)|=\kappa _v(\Gamma)$ will be called a $v$--fragment.

\begin{lemma} Let $F_1$ and $F_2$ be two $v$--fragments.
Then $F_1\cap F_2$ and  $F_1\cup F_2$ are $v$--fragments.
In particular, there exists a $v$--fragment $K_v$ contained in  every $v$--fragment.
\end{lemma}

\begin{proof}
Notice that the intersection
and union of two $v$--Moser sets are $v$--Moser sets.
By (\ref{submodularity}),
$2\kappa _{\Gamma}(v) \le |\partial  (F_1 \cap F_2)|+|\partial  (F_1 \cup F_2)|\le |\partial  (F_1 )|+|\partial ( F_2)|\le 2\kappa _{\Gamma}(v).$
\end{proof}

We shall denote the minimal $v$--fragment containing  $v$ by $K_v$.

Consider the subgraph $\Theta =(V,F),$ where $F=\{(x,y)\in E :  y\in K_x\}$.
 Set $\Psi (x) =\{y :  K_y \subset K_x \ \mbox{and}\  y\in \Theta (x)\}$.

\begin{lemma} \label{mader}
$\Theta^-(x)\setminus \Psi (x) \subset \G(K_x)\setminus \Gamma (x)$.
\end{lemma}

\begin{proof}

Take $y\in \Theta^-(x)\cap (V\setminus \Gamma (K_x))$. Since   $x\in K_y,$ we have by (\ref{id0}), $\Gamma ^-(y)\subset
\Gamma ^-(V\setminus \Gamma (K_x))\subset V\setminus K_x$. It follows that  $\Gamma ^-(y)\cap (K_x \cup K_y)=\Gamma ^-(y)\cap  K_y=\{y\}.$ Thus $K_x \cup K_y$ is a $y$--Moser set.
Noticing that $K_x \cap K_y$ is an $x$--Moser set,  we have  by (\ref{submodularity}),

$\kappa _{\Gamma}(y)+ \kappa _{\Gamma} (x)\le |\partial  (K_x \cap K_y)|+|\partial  (K_x \cup K_y)|\le |\partial  (K_x )|+|\partial ( K_y)|\le \kappa _{\Gamma}(x)+ \kappa _{\Gamma} (y).$

It follows that $K_x \cap K_y$ is a $x$--Moser set and hence $K_x=K_x \cap K_y$. Thus $y\in \Psi(x).$ Therefore $$\Theta^-(x)\cap (V\setminus \Gamma (K_x))\subset \Psi (x).$$

Take  $y\in \Theta^-(x)\setminus \Psi (x)$. Then $x\in \Theta (y)\subset \Gamma (y).$  By the last relation, we must have $y\in \Gamma (K_x)$.
Since $K_x\cap \Gamma ^-(x)=\{x\},$ we have $y\notin
\Gamma (x)$. In particular $\Theta^-(x)\setminus \Psi (x)\subset \Gamma(x)$.
   The lemma follows now.
\end{proof}

\section{The result and its applications}

\begin{theorem}\label{main}
Let $\Gamma =(V,E)$  be a point-transitive reflexive locally finite
relation and  let $v\in V.$
Let $F$  be a  subset of $V$ with $F\cap \Gamma ^-(v)=\{v\}$.
\begin{itemize}
  \item[(i)] If $F$  is finite, then
$| \Gamma (F)\setminus F|\ge
|\Gamma (v)|-1.$
\item [(ii)]If $F$  is cofinite, then
$| \Gamma (F)\setminus F|\ge
|\Gamma ^- (v)|-1.$
\end{itemize}

\end{theorem}
\begin{proof}
Recall that for every $x\in V,$ $|\Gamma (v)|=|\Gamma (x)|$. This is an immediate consequence
of the transitive action of automorphism group. Moreover if $V$ is finite,
$|V||\Gamma (v)|=|E|=|V||\Gamma^- (v)|$.
Thus $|\Gamma (v)|=|\Gamma ^- (v)|.$  In particular, (i) and (ii) are equivalent in the finite case.
For every $x\in V,$ there is an automorphism $\phi$ such that $\phi (v)=x.$ Clearly
$\phi (K_v)= \phi (K_x)$ and hence
$|K_v|=|K_x|.$ If $x\in \Gamma (v),$  then $v\in K_v\setminus K_x,$ and thus we have $K_v \neq K_x$ and hence $K_v\not\subset
K_x$. It follows that $\Psi (x)=\{x\}.$

{\bf Case 1} $V$ is finite.
Since $\sum _{x\in V}|\Psi ^{-}(x)|=\sum _{x\in V}|\Psi (x)|,$
there is a $u\in V$ with $|\Theta(u)|\le  |\Theta^-(u)|$.

 By Lemma \ref{mader},
 $\Theta^-(u) \setminus \{u\} =\Theta^-(u)\setminus \Psi (u)\subset \G(K_u)\setminus \Gamma (u)$. Now we have
\begin{eqnarray*}|\G(K_u)| &\ge& |\Theta ^-(u)\setminus \{u\} |+| \Gamma (u) \setminus K_u|\\&\ge& |\Theta(u)\setminus \{u\}|
+| \Gamma (u)\setminus K_u|\\&\ge& |(\Gamma(u)\cap K_u)\setminus \{u\}|
+| \Gamma (u)\setminus K_u|= |\Gamma (u)|-1.\end{eqnarray*}
Since $ \{u\}$ is a Moser set, we have $|K_u|=1$. It follows that $|K_v|=|K_u|=1.$
Since $F$ is a Moser set, we have $$
| \Gamma (F)\setminus F|\ge  \kappa _v (\Gamma )=
|\Gamma (K_v)|-1=|\Gamma (v)|-1.$$
 The result follows in this case.

{\bf Case 2} $V$ is infinite.

{\bf Subcase 2.1} $F$ is finite.

Let $A$ be an atom of $\Gamma$ containing $v.$ Assume first that $\kappa (\Gamma)=0.$
Every point  $x\in V\setminus A$ is contained in an atom $A_x,$ disjoint from $A$ by Proposition \ref{atom} with
$\partial (A_x)|=0.$ It follows that $\Gamma (F\setminus v)\cap A=\emptyset$. Since
$A$ is finite the result holds for $A$ by Case 1. It follows that $|\Gamma (F)|=|\Gamma (F\setminus A)|+
|\Gamma (F\cap A)|\ge |F\setminus A|+|F\cap A|+|\Gamma (v)|-1.$ So we may assume that $\kappa (\Gamma)>0.$
Observe that $F\cap A$ is a $v$--Moser set of $A$. By case 1, $$
| \Gamma (F\cap A)|\ge |A\cap F|+|A\cap \Gamma (v)|-1.$$
 By the definition of $\kappa,$ we have $|\partial (F\cup,A)|\ge \kappa=\partial (A).$
Notice that $\partial (F\cup A)\setminus \partial (F)\subset \partial (A) \setminus \Gamma (v)$,
and hence \begin{eqnarray*}|\partial (F)|&\ge& | \G (F)\cap A|+ |\partial (F\cap A)\cap \partial (F)|\\ &\ge&
| \G (F)\cap A|+ |\partial (F\cup A)|-|\partial (A) \setminus
 \Gamma (v)|\\&\ge& |A\cap \Gamma (v)|-1+|\partial (A)|-|\partial (A)\setminus \Gamma (v)|\\&=& |A\cap \Gamma (v)|-1+|\partial (A)\cap \Gamma (v)|=|\Gamma (v)|-1.\end{eqnarray*}

 {\bf Subcase 2.2} $F$ is cofinite.
 Put $R=(V\setminus \Gamma (F))\cup \{v\}.$ It follows that $R\cap \Gamma (v)=\{v\}\cap \Gamma (v)=\{v\}.$ In particular, $R$ is a $v$--Moser set for the relation
 $\Gamma ^-$. Also we have using (\ref{id0}), $\G ^-(R)\subset \G^-(V\setminus \Gamma (F))\cup (\G^-(v)\setminus R)\subset \G(F),$ since $F\cap \Gamma ^-(v)=\{v\}.$

 By Subcase 2.1, $|\G (F)|\ge |\G^-(V\setminus \Gamma (F))|\ge
 |\G^-(R)|\ge |\Gamma^- (v)|-1.$
\end{proof}

\begin{corollary}\label{sphere}
Let $\Gamma =(V,E)$  be a point-symmetric reflexive locally finite relation and  let $v\in V.$
Let $j\ge 1$ be an integer such that  $\Gamma ^{j-1}(v)\cap \Gamma ^{-}(v)=\{v\}$. Then
 $$
|\Gamma ^{j}  (v)|\ge | \Gamma ^{j-1}  (v)| +
|\Gamma (v)|-1.$$
\end{corollary}

The proof is an immediate consequence of Theorem \ref{main}.

The finite case of the  corollary is proved by Mader in \cite{mader}.
The same conclusion was obtained under the stronger hypothesis $\Gamma ^j(v)\cap \Gamma ^{-}(v)=\{v\}$ in \cite{Hspheres}. This weak form is enough to show Caccetta and H\"aggkvist Conjecture.

\begin{corollary}\label{mskw} (The Moser-Scherck-Kemperman-Wehn Theorem and its extension to cofinite sets)

Let $G$  be group and let $S$ be a finite subset of $G.$ Let $F$ be a finite or a cofinite subset of $G$ such that  $F\cap  S^{-1}=\{1\}$.
Then
$| (FS)\setminus F|\ge
|S|-1.$

\end{corollary}
\begin{proof}
The corollary follows by applying Theorem \ref{main} to the Cayley relation
defined on $G$ by $S$.\end{proof}

%subsection{Terminology}
Let $G$ be a group and let  $X\subset G.$ We shall write $\overline{X}=G\setminus X$ and
$\tilde{X}=X\cup \{1\}.$

\begin{corollary}\label{ce} (Losonczy in the abelian case,   Eliahou-Lecouvey in the general case)

Let $G$  be group and let $A$ be a finite subset of $G\setminus \{1\}.$  There is a permutation $\sigma$ of $S$ such that $x\sigma (x)\notin A$ for every
$x\in A.$

\end{corollary}
\begin{proof}
Define a relation $\Phi$ on $ A,$ where $\Phi (x)=(x^{-1}\oa)\cap A.$
For every $B\subset A^{-1},$ we have clearly $\oa\cap \tilde{B}=\{1\}.$
Since $\oa$ is a cofinite subset, we have by Corollary \ref{mskw},
$|B|\le |(\oa \tilde{B})\setminus \oa|=|(\oa B)\setminus \oa|=|\Phi (B^{-1})|.$
The existence $\sigma$ follows now by Theorem \ref{kh}.
\end{proof}

\begin{corollary}(Mader \cite{mader})\label{maderth}
Let $\Gamma =(V,E)$  be a point-transitive loopless  finite
relation and  let $v\in V.$ There exist $|\Gamma (v)| $ directed cycles in a with a  pairwise intersection $=\{v\}$.
\end{corollary}
\begin{proof} Let $\Phi $ be the  graph obtained by a adding vertex $  v'\notin V,$ with $\Phi(v')=\Gamma ^- (v).$ Let $\Theta $ be the reflexive
closure $\Phi$ (obtained by adding loops everywhere. Let $F$ be a subset of $V$
with $v\in F$ and $v'\notin \Theta (F).$ Then clearly $F$ is a $v$--Moser set.
By Theorem \ref{main}, $|\G(F)|\ge |\Theta (v)|.$ By the Dirac-Menger's Theorem \ref{menger},
$\Phi$ contains $|\Gamma (v)|$ disjoint paths from $\Gamma (v)$ to $\Gamma ^-(v')=\Gamma ^-(v)$.
Adding $v$ to each of these paths, we see the existence of $|\Gamma (v)|$ cycles verifying the Corollary.
\end{proof}

%\end{document}

\end{document}